\documentclass[12pt]{article}
\usepackage{amsfonts}
\usepackage{amsmath}
\usepackage{mathrsfs}
\usepackage{amssymb}
\usepackage{color,amscd,array,graphicx,mathdots,bm,amsfonts,epsfig,amsmath}
\usepackage{amsthm}
\usepackage{enumerate}
\usepackage{geometry}
\newtheorem{theorem}{Theorem}[section]
\newtheorem{corollary}{Corollary}[section]

\newtheorem{proposition}{Proposition}[section]

\newcommand{\levy}{L\'{e}vy}

\newcommand{\cadlag}{ c\`agl\`ad}

\makeatletter
\newcommand{\figcaption}{\def\@captype{figure}\caption}
\newcommand{\tabcaption}{\def\@captype{table}\caption}
\makeatother

\newfam\msbmfam\font\tenmsbm=msbm10\textfont
\msbmfam=\tenmsbm\font\sevenmsbm=msbm7
\scriptfont\msbmfam=\sevenmsbm

\def\PP{\mathbb P}
\def\RR{\mathbb R}

\def\cN{{\cal N}}

\def\<{\left<}\def\>{\right>}

\def\({\left(}\def\){\right)}

\begin{document}
\title{Pathwise uniqueness for stochastic differential equations driven by pure jump processes \footnote{This research is supported by Macao Science and Technology Fund FDCT
025/2016/A1.}
}
\author{Jiayu Zheng and Jie Xiong \\ }
\date{}
\maketitle
\begin{abstract}
Based on the weak existence and weak uniqueness, we study the pathwise uniqueness of the solutions for a class of one-dimensional stochastic differential equations driven by pure jump processes. By using Tanaka's formula and the local time technique, we show that there is no gap between the strong uniqueness and weak uniqueness when the coefficients of the Poisson random measures satisfy a suitable condition.\bigskip
\end{abstract}

\noindent{\bf Keywords:} Pure jump process, weak uniqueness, Tanaka's formula, pathwise uniqueness, local time.

\noindent{\bf MSC:} Primary 60H10; secondary 60H30.

\section{Introduction}
In the study of stochastic equations, it is common to distinguish between strong uniqueness and weak uniqueness of solutions. Roughly speaking, strong uniquenss asserts that two solutions on the same probability space with the same stochastic inputs agree almost surely while weak uniqueness asserts that two solutions agree in distribution. Yamada and Watanabe \cite{YW} proved that weak existence and strong uniqueness imply strong existence and weak uniqueness. Engelbert \cite{ES} extended this result to a somewhat more general class of equations and gave a converse in which the roles of existence and uniqueness are reversed, that is, weak uniqueness and strong existence imply strong uniqueness. A general version of the Yamada-Watanabe and Engelbert results which apply to a wide variety of stochastic equations was given by Kurtz \cite{K}.

We study the solutions of the stochastic differential equations (SDEs) driven by pure jump processes. Suppose that $U_0$ and $U_1$ are complete separable metric spaces,  and that $\mu_0(du)$ and $\mu_1(du)$ are $\sigma$-finite Borel measures on $U_0$ and $U_1$, respectively. The coefficients $g_0(x,u)$ and $g_1(x,u)$ are Borel functions on $\mathbb{R} \times U_0$ and $\mathbb{R} \times U_1$, respectively, which have at most countably many discontinuous points, and satisfy the following condition:

 (1.a) For any fixed $u$, $g_i(x,u)+x$, $i=0,\ 1$, are non-decreasing in $x\in\RR$.

Let $\{p_0(t)\}$ and $\{p_1(t)\}$ be $(\mathscr{F}_t)$-Poisson point processes on $U_0$ and $U_1$ with characteristic measures $\mu_0(du)$ and $\mu_1(du)$, respectively. Suppose that $\{p_0(t)\}$ and $\{p_1(t)\}$ are independent of each other.

Let $N_0(ds, du)$ and $N_1(ds, du)$ be the Poisson random measures associated with $\{p_0(t)\}$ and $\{p_1(t)\}$, respectively. Let $\widetilde{N}_0(ds, du)$ be the compensated measure of $N_0(ds, du)$. By a solution of the SDE
\begin{align}\label{eq1}
X_t = X_0 + \int_0^t b(X_{s})ds + \int_0^t \int_{U_0} g_0(X_{s-}, u) \widetilde{N}_0(ds,du) + \int_0^t \int_{U_1} g_1(X_{s-}, u) N_1(ds, du),
\end{align}
we mean a {\cadlag} and $(\mathscr{F}_t)$-adapted process $\{X_t\}$ that satisfies the equation almost surely for every $t\ge 0$.

In this work we consider the pathwise uniqueness for the solution of (\ref{eq1}). We assume that the weak existence and the weak uniqueness have already been known. Actually, this work is motivated by an unsolved problem in Etheridge and Kurtz \cite{EK} where a system of stochastic equations arisen from a population genealogic model. Weak existence of solutions of the system follows by approximation, while weak uniqueness holds for every single euqation in the system. They point out that such weak uniqueness for a single equation does not imply weak uniqueness for the system while strong uniqueness for each single equation would give strong uniqueness for the system. They proved the uniqueness of the system of SDEs when there are only finite many jumps in any finite time interval. However, the uniqueness problem for the system of SDEs under the weaker conditions which allows infinitely many very small jumps in a finite time interval was left open.  Since the weak uniqueness holds for every equation, we thus proceed to considering its strong uniqueness. Inspired by the work of Zheng $et$ $al$ \cite{ZXZ}, we prove the strong uniqueness for the case of $d=1$, where $d$ is the spatial dimension in their model.

In the book \cite{DM} of Revuz and Yor, it is proved that the weak uniqueness together with the local time condition, that is, the local time at 0 of the difference of two solutions vanishes, then the pathwise uniqueness holds. They considered the equations driven by continuous semimartingales. Zheng $et$ $al$ \cite{ZXZ} extended this result to a more general class of SDEs with jumps. From the above statement, it is obvious that there is a gap is between the weak uniqueness and strong uniqueness for general stochastic equation. What may seem surprising in the results we are about to prove is that we get strong uniqueness directly when the weak uniqueness holds  for SDEs driven by pure jump processes. One can also observe that in this case the maximum of any two solutions of the equation remains a solution, which is usually not true for SDEs driven by Brownian motion due to the appearance of a local time.

This paper is arranged as follows: In Section 2 we prove the pathwise uniqueness by using Tanaka's formula and the local time technique. Then an example from \cite{EK}  which motives our research is discussed in Section 3.

\section{Pathwise uniqueness}
In this section, we discuss the pathwise uniqueness of the equation (\ref{eq1}) based on its weak existence and weak uniqueness.

We first state Tanaka's formula and an important Corollary whose proofs can be found on page 219 of Protter \cite{PE}. These two results are then used to prove that the maximum of two solutions of (\ref{eq1}) is also a solution. Finally, combining with the weak uniqueness, we obtain the strong uniqueness of the solution to (\ref{eq1}).

{\em Tanaka's Formula}: Let $X$ be a semimartingale and let $L^a$ be its local time at $a\in\RR$. Then,
\begin{align*}
(X_t - a)^{+} - (X_0 - a)^{+} =& \int_{0+}^t \mathbf{1}_{\{X_{s-} >a\}} dX_s + \sum_{0<s\le t}\mathbf{1}_{\{X_{s-} >a\}} (X_s - a)^{-}  \nonumber \\
&+\sum_{0<s\le t}\mathbf{1}_{\{X_{s-} \le a\}} (X_s - a)^{+} + \frac{1}{2} L_t^a.
\end{align*}

\begin{corollary} \label{coro1}
Let $X$ be a semimartingale with local time $(L^a)_{a \in \mathbb{R}}$. Let $g$ be a bounded Borel measurable function. Then,
\[
\int_{-\infty}^{\infty} L_t^a g(a)da = \int_0^t g(X_s) d[X,X]_s^c,  \quad a.s.
\]
\end{corollary}

The following result is the key of this article.
\begin{proposition} \label{prop1}
If $X^1$ and $X^2$ are two solutions of (\ref{eq1}) such that $X_0^1 = X_0^2$ a.s., then  $X^1 \vee X^2$ is also a solution of (\ref{eq1}).
\end{proposition}
\begin{proof}
We only need verify that $X^1_t \vee X^2_t$ satisfies (\ref{eq1}) for $t\le T$ for any fixed $T$. Denote $X=X^2-X^1$. It is clear that
 $[X, X]^c = 0.$ Applying Corollary {\ref{coro1}} to $g(a) =1$,  we have
$$ \int_{-\infty}^{\infty} L_T^a da = 0  \quad a.s.,$$
and hence,
\[\int_{-\infty}^{\infty} \mathbb{E} L_T^a da = 0.\]
Let $\mathcal{N} = \{a\in\RR : \mathbb{E} L_T^a \ne 0\}$. Then, $Leb(\mathcal{N}) = 0$. Since $L^a_t$ is non-decreasing in $t$, for any $a\in\cN^c$, we have
 \[\PP\(L^a_t=0,\;\;\forall t\le T\)=1.\]
Now we choose a sequence $\{a_n\}\subset\cN^c$ decreasing to 0.

Since $X=X^2-X^1$ is a semimartingle, by Tanaka's formula, we have
\begin{eqnarray}  \label{eq1234}
&& X_t^1  \vee (X_t^2 - a_n)    
= X_t^1 + (X_t - a_n)^{+}  \nonumber \\
&=&  X_t^1 +\int_{0+}^t \mathbf{1}_{(X_{s-} > a_n)} dX_s + \sum\limits_{0 < s \le t} \mathbf{1}_{(X_{s-} > a_n)} (X_s -a_n) ^{-}  \nonumber\\
&&+  \sum\limits_{0 < s \le t} \mathbf{1}_{(X_{s-} \le  a_n )} (X_s - a_n) ^{+},\qquad a.s.
\end{eqnarray}
Letting $n$ tends to infinity, we have
\begin{align}   \label{eq12345}
X_t^1  \vee X_t^2
=&  X_t^1 +\int_{0+}^t \mathbf{1}_{(X_{s-} >0)} dX_s + \sum\limits_{0 < s \le t} \mathbf{1}_{(X_{s-} >0 )} X_s^{-}  \nonumber\\
&+  \sum\limits_{0 < s \le t} \mathbf{1}_{(X_{s-} \le 0)} X_s ^{+} .
\end{align}
Denote by $D_0$ and $D_1$ all jumping times of $N_0$ and $N_1$, respectively. Then
\begin{eqnarray*}
&&X_s =  X_{s-}+\Delta X_s  \\
&=&     X_{s-}+\Big( g_0\left( X_{s-}^2, p_0(s)\right) - g_0\left (X_{s-}^1,  p_0(s)\right) \Big)\mathbf{1}_{D_0}(s) \\
&&+  \Big(g_1 \left(X_{s-}^2,  p_1(s)\right) - g_1 \left(X_{s-}^1, p_1(s)\right)\Big)\mathbf{1}_{D_1}(s) .
\end{eqnarray*}
Since $g_0(x,u)$ and $g_1(x,u)$ satisfy Condition (1.a), then
\begin{eqnarray}
\mathbf{1}_{(X_{s-}>0)} X_s^{-} 
&=& \(X_{s-}^2 - X_{s-}^1\)^-\mathbf{1}_{(X_{s-}^2 >X_{s-}^1)}\mathbf{1}_{(D_0\cup D_1)^c}(s)\nonumber\\
 &&+\(  X_{s-}^2 - X_{s-}^1+g_0\left( X_{s-}^2, p_0(s)\right) - g_0\left (X_{s-}^1,  p_0(s)\right)\)^-
\mathbf{1}_{(X_{s-}^2 >X_{s-}^1)}\mathbf{1}_{D_0}(s)\nonumber\\
&& +\(  X_{s-}^2 - X_{s-}^1+g_1 \left(X_{s-}^2,  p_1(s)\right) - g_1 \left(X_{s-}^1, p_1(s)\right)\)^-
\mathbf{1}_{(X_{s-}^2 >X_{s-}^1)}\mathbf{1}_{D_1}(s)\nonumber\\
&=& 0  ,   \nonumber
\end{eqnarray}
and hence,
\begin{equation}\label{eq0326a}
 \sum\limits_{0 < s \le t} \mathbf{1}_{(X_{s-}>0)} X_s^{-}  =0.
\end{equation}
Similarly, 
\begin{equation}\label{eq0326b}
\sum\limits_{0 < s \le t} \mathbf{1}_{(X_{s-} \le 0)} X_s^{+} = 0. \end{equation}
Plugging (\ref{eq0326a}) and (\ref{eq0326b}) into (\ref{eq12345}), we get
\begin{equation}\label{eq0326c}
X_t^1  \vee X_t^2
=  X_t^1 +\int_{0+}^t \mathbf{1}_{(X_{s-} >0)} dX_s.
\end{equation}

Replacing $X^{i}_s$ by
\begin{align*}
X_0^i + \int_0^t b(X^{i}_{s})ds + \int_0^t \int_{U_0} g_0(X^i_{s-}, u) \widetilde{N}_0(ds, du) + \int_0^t \int_{U_1} g_1(X^i_{s-}, u) N_1(ds, du), \quad i = 1,2,
\end{align*}
we can continue (\ref{eq0326c}) with
\begin{align}  \label{eq3}
X_t^1 \vee X_t^2
=& X_0^1 + \int_0^t b(X^1_{s})ds + \int_0^t \int_{U_0} g_0(X^1_{s-}, u) \widetilde{N}_0(ds, du) + \int_0^t \int_{U_1} g_1(X^1_{s-}, u) N_1(ds, du)  \nonumber  \\
&+ \int_0^t \mathbf{1}_{(X_{s}^2 >X_{s}^1)}(b(X^{2}_{s})- b(X^{1}_{s}))ds \nonumber  \\
&+ \int_{0+}^t \int_{U_0} \mathbf{1}_{(X_{s-}^2 >X_{s-}^1)} (g_0(X^2_{s-}, u)  - g_0(X^1_{s-}, u) ) \widetilde{N}_0(ds, du)    \nonumber  \\
&+ \int_{0+}^t \int_{U_1} \mathbf{1}_{(X_{s-}^2 >X_{s-}^1)} (g_1(X^2_{s-}, u)  - g_1(X^1_{s-}, u) ) N_1(ds, du).
\end{align}
Simplifying (\ref{eq3}), we have
\begin{align*}
X_t^1 \vee X_t^2
=& X_0^1 \vee X_0^2 + \int_0^t \(\mathbf{1}_{(X_{s}^2 \le X_{s}^1)}b(X^1_{s}) +\mathbf{1}_{(X_{s}^2 >X_{s}^1)}b(X^2_{s})\)ds \\
& + \int_0^t \int_{U_0} \(\mathbf{1}_{(X_{s-}^2 \le X_{s-}^1)}g_0(X^1_{s-},u) +\mathbf{1}_{(X_{s-}^2 >X_{s-}^1)}g_0(X^2_{s-},u)\)  \widetilde{N}_0(ds, du) \\
&+\int_0^t \int_{U_1} \(\mathbf{1}_{(X_{s-}^2 \le X_{s-}^1)}g_1(X^1_{s-},u) +\mathbf{1}_{(X_{s-}^2 >X_{s-}^1)}g_1(X^2_{s-},u)\)  N_1(ds, du) \\
=& X_0^1 \vee X_0^2 + \int_0^t b(X^1_{s} \vee X^2_{s})ds + \int_{0+}^t \int_{U_0}  g_0(X^1_{s-} \vee X^2_{s-}, u)  \widetilde{N}_0(ds, du) \\
&+ \int_{0+}^t \int_{U_1} g_1(X^1_{s-} \vee X^2_{s-}, u)   N_1(ds, du),
\end{align*}
which implies that $X_t^1 \vee X_t^2$ is also a solution of (\ref{eq1}).
\end{proof}

\begin{theorem}
Suppose that the equation (\ref{eq1}) has a weak solution. Then the weak uniqueness for the equation (\ref{eq1}) implies the pathwise uniqueness.
\end{theorem}
\begin{proof}
Suppose that $X_1$ and $X_2$ are two solutions of the SDE (\ref{eq1}). From Proposition \ref{prop1} we know that $X^1 \vee X^2$ is also a solution of (\ref{eq1}).

Since the weak uniqueness holds for the equation (\ref{eq1}), we have \[ \mathcal{L}(X_t^1) = \mathcal{L}(X_t^2) = \mathcal{L}(X_t^1 \vee X_t^2), \]
where $\mathcal{L}(\xi)$ denotes the law of the random variable $\xi$. Then, $\mathbb{E} (X_t^1 \vee X_t^2 - X_t^1) = 0$. Since  $X_t^1 \vee X_t^2 - X_t^1$ is a non-negative random variable, we get $X_t^1 \vee X_t^2 = X_t^1$ a.s. Similarly, we have $X_t^1 \vee X_t^2 = X_t^2$  a.s., which implies $X_t^1 = X_t^2$  a.s.
\end{proof}

\section{Application}
In \cite{EK} Etheridge and Kurtz represented the $\Lambda$-Fleming-Viot process as the limit of an infinite system of particles governed by SDE (4.25) (in their paper) driven by Poisson random measures. The goal of this section is to establish the pathwise uniqueness of the solution to this SDE by making use of our results proved in last section.

For the convenience of the reader, we now introduce Etheridge-Kurtz equation (namely, SDE (3.7) below) briefly. Let $D_{y,w} \subseteq \mathbb{R}^d$ be the ball centered at $y$ with radius $w$. Let $E = \{0,1\} \times D_{0,1} \times \mathbb{R}^d \times [0,1] \times [0, \infty)$, and let $\xi$ be a Poisson random measure on $[0, \infty) \times E$ with mean measure
\[
ds((1 - \zeta)\delta_0(\theta) + \zeta \delta_1(\theta)) v_{0,1}(dv)dy \nu^1(w, d\zeta)\nu^2(dw),
\]
where $v_{0,1}$ is the uniform distribution on the ball $D_{0,1}$, $\nu^2$ is a measure on $[0, \infty)$ and $\nu^1$ is a transition measure from $[0, \infty)$ to $[0, 1]$ satisfying the following conditions
\begin{align} \label{eqc1}
\int_{[0,1]\times(1,\infty)} \zeta w^d \nu^1(w, d\zeta)\nu^2(dw) <\infty
\end{align}
and
\begin{eqnarray}\label{eqc2}
\begin{cases}
\int_{[0,1]\times [0,1]} \zeta |w|^2 \nu^1(w, d\zeta)\nu^2(dw) <\infty, &if \quad d = 1,
\cr \int_{[0,1]\times [0,1]} \zeta |w|^{2+d} \nu^1(w, d\zeta)\nu^2(dw) <\infty, &if \quad d \ge 2.
\end{cases}
\end{eqnarray}
The above conditions (\ref{eqc1}) and (\ref{eqc2}) imply the existence of the stochastic integrals in the limiting equation (\ref{eqA2}) below, while \cite{VW} assumes
\begin{align} \label{eqc3}
\int_{[0,1]\times(0,\infty)} \zeta w^d \nu^1(w, d\zeta)\nu^2(dw) <\infty.
\end{align}
Define
\[
\Gamma_k = D_{0,k} \times [0,1] \times [2^{-k}, 2^{k}].
\]
Denote $z = (\theta, v, y, \zeta, w)$ and $\xi(ds, dz) = \xi(ds, d\theta, dv, dy, d\zeta, dw) $. The following is the SDE we will study in this section:
\begin{align} \label{eqA2}
X(t) & =  X(0) + \lim_{k \to \infty} \int_{[0,t]\times \{0,1\}\times D_{0,1}\times \Gamma_k} \mathbf{1}_{D_{y,w}} (X(s-))\theta \left( y + wv- X(s-) \right) \xi (ds, dz) \nonumber  \\
& = X(0) + \int_{[0,t] \times E}  \mathbf{1}_{D_{y,w}} (X(s-))\theta \left( y + wv- X(s-) \right) \tilde{\xi} (ds, dz),
\end{align}
where $\tilde{\xi}$ is $\xi$ centered by its mean measure.

The requirement in (\ref{eqc3}) implies that a point in space is involved in a birth/death event only finitely often in a finite time interval while the weaker conditions (\ref{eqc1}) and (\ref{eqc2}) allows infinitely many very small birth/death events. If (\ref{eqc3}) holds, strong uniqueness has been proved in Lemma 4.3 of \cite{EK}. Now we are set about to prove the strong uniqueness under conditions (\ref{eqc1})  and (\ref{eqc2}) for the case $d=1$.

\begin{theorem}
The SDE (\ref{eqA2}) has a unique strong solution.
\end{theorem}
\begin{proof}
As they pointed out in Lemma 4.3 of \cite{EK}, weak existence for SDE (\ref{eqA2}) follows by approximation and the weak uniqueness follows by uniqueness of the corresponding martingale problem ($X$ is a {\levy} process). To prove the pathwise uniqueness, we only need to verify the condition $(1.a)$ for $g_0$ defined by $$g_0(x,z) = \mathbf{1}_{D_{y,w}} (x)\theta \left( y + wv- x \right).$$ Now, we consider the monotonicity of $x + g_0(x,z)$ in $x$.

When $x \in D_{y,w}$, $x + g_0(x,z) = x + \theta(y + wv - x) = (1 - \theta)x + \theta(y + wv)$, which is non-decreasing in $x$.

When $x \notin D_{y,w}$, $x + g_0(x,z) = x$. Then $\forall x \in [y+w, \infty)$, $x + g_0(x,z)$ is non-decreasing in $(y-w, \infty)$ if and only if
\[(1 - \theta)(y+w) + \theta(y+wv) \le y+w .\]
i.e. $$(1- \theta)w + \theta wv \le w \Leftrightarrow (v-1)\theta \le 0, $$
$(v-1)\theta \le 0$ does hold since $v \in [-1,1], \theta \in \{0,1\}.$

Similarly, it is easily to check
$$(1 - \theta)(y-w) + \theta(y-wv) \ge  y-w , \quad \forall x \in (-\infty, y-w) $$
which implies $x+g_0(x,z)$ is non-decreasing in $(-\infty ,y+w)$. Consequently, $x+g_0(x,z)$ is non-decreasing in $x$.

By Proposition 2.1 and Theorem 2.2, we conclude that (\ref{eqA2}) satisfies strong uniqueness under conditions (\ref{eqc1}) and the case of $d=1$ in  (\ref{eqc2}). Hence, the strong uniqueness holds for the system.
\end{proof}

{\bf Acknowledgment} We would like to thank Tom Kurtz who posed this problem to us and offered many suggestions.

\end{document}